\documentclass[11pt,twoside]{article}
\usepackage{mathrsfs}
\usepackage{amsmath}
\usepackage{amsthm}
\input amssymb.sty
\usepackage{amsfonts}
\usepackage{amssymb}
\usepackage{latexsym}
\usepackage[all]{xy}

\date{\empty}
\allowdisplaybreaks[4] \footskip=15pt

\numberwithin{equation}{section} \theoremstyle{plain}
\newtheorem*{thm*}{Main Theorem}
\newtheorem{theorem}{Theorem}[section]
\newtheorem{corollary}[theorem]{Corollary}
\newtheorem*{corollary*}{Corollary}

\newtheorem*{claim*}{Claim}
\newtheorem{lemma}[theorem]{Lemma}
\newtheorem*{lemma*}{Lemma}
\newtheorem{proposition}[theorem]{Proposition}
\newtheorem*{proposition*}{Proposition}

\newtheorem*{remark*}{Remark}

\newtheorem*{example*}{Example}

\newtheorem*{question*}{Question}

\newtheorem*{definition*}{Definition}

\newtheorem*{acknowledgements*}{ACKNOWLEDGEMENTS}

\begin{document}
\begin{center}
{\large  \bf The Moore-Penrose inverse of differences and products of projectors in a ring with involution}\\
\vspace{0.8cm} {\small \bf  Huihui ZHU$^{[1,2]}$, Jianlong CHEN$^{[1]}$\footnote{Corresponding author.

1 Department of Mathematics, Southeast University, Nanjing 210096, China.

2 CMAT-Centro de Matem\'{a}tica, Universidade do Minho, Braga 4710-057, Portugal.

3 Departamento de Matem\'{a}tica e Aplica\c{c}\~{o}es, Universidade do Minho, Braga 4710-057, Portugal.

Email: ahzhh08@sina.com(H. ZHU), jlchen@seu.edu.cn(J. CHEN), pedro@math.uminho.pt (P. PATR\'{I}CIO).}, Pedro PATR\'{I}CIO$^{[2,3]}$ }
\end{center}

\bigskip

{ \bf  Abstract:}  \leftskip0truemm\rightskip0truemm  In this paper, we study the Moore-Penrose inverses of differences and products of projectors in a ring with involution. Also, some necessary and sufficient conditions for the existence of such inverses are given, and their expressions are presented.


\bigskip


\section { \bf Introduction}
Throughout this paper, $R$ is a unital $*$-ring, that is a ring with unity 1 and an involution $a\mapsto a^*$ satisfying that $(a^*)^*=a$, $(a+b)^*=a^*+b^*$, $(ab)^*=b^*a^*$. Recall that an element $a\in R$ is said to have a Moore-Penrose inverse (abbr. MP-inverse) if there exists $b\in R$ such that the following equations hold \cite{P}:
\begin{center}
$aba=a$, $bab=b$, $(ab)^*=ab$, $(ba)^*=ba$.
\end{center}
Any $b$ that satisfies the equations above is called a MP-inverse of $a$. The MP-inverse of $a\in R$ is unique if it exists and is denoted by $a^\dag$. By $R^\dag$ we denote the set of all MP-invertible elements in $R$.

MP-inverse of differences and products of projectors in various sets attracts wide attention from many scholars. For instance, Cheng and Tian \cite{CT} studied the MP-inverses of $pq$ and $p-q$, where $p$, $q$ are projectors in complex matrices. Li \cite{L} investigated how to express MP-inverses of product $pq$ and differences $p-q$ and $pq-qp$, for two given projectors $p$ and $q$ in a $C^*$-algebra. Later, Deng and Wei \cite{DW1} derived some formulae for the MP-inverse of the differences and the products of projectors in a Hilbert space. Recently, Zhang et al. \cite{ZZCW} obtained the equivalences for the existences of differences and products of projectors in a $*$-reducing ring. More results on MP-inverses can be found in \cite{KDC,KP,P}.

Motivated by \cite{KRS}, we investigate the equivalences for the existences of the MP-inverse of differences and products of projectors in a ring with involution. Moreover, the expressions of the MP-inverse of differences and products of projectors are presented. Some known results in $C^*$-algebras are extended.

Note that neither dimensional analysis nor special decomposition in Hilbert spaces and $C^*$-algebras can be used in rings. The results in this paper are proved by a purely ring theoretical method.

\section{Some lemmas}

~~~~We begin with some lemmas which play an important role in the sequel.

In 1992, Harte and Mbekhta \cite{HM} showed the excellent result in $C^*$-algebras that if $a$ is MP-invertible, then $a^*c=ca^*$ and $ac=ca$ imply $a^\dag c =c a^\dag$. More precisely, it follows from \cite[Corollary 12]{Mary} that in a $*$-semigroup $a^\dag \in {\rm comm}^2\{a,a^*\}$, i.e., $a^\dag$ double commutes with $a$ and $a^*$. In 2013, Drazin \cite{D1} then proved the following.

\begin{lemma} {\rm \cite[Corollary 2.7]{D1}} Let S be any $*$-semigroup, let $a_1, a_2, d \in S$, and suppose that $a_1$ and
$a_2$ each have Moore-Penrose inverses $a_1^\dag$, $a_2^\dag$, respectively. Then, for any $d \in S$, $da_1 = a_2d$
and $da_1^*=a_2^*d$ together imply $a_2^\dag d=da_1^\dag$.
\end{lemma}

The following result in $C^*$-algebras was considered by Koliha \cite{K}. For the convenience of the reader, we give its proof in a ring.

\begin{lemma} \label{ab,a'b} Let $a,b\in R^\dag$ with $ab=ba$ and $a^*b=ba^*$. Then $ab\in R^\dag$ and $(ab)^\dag= b^\dag a^\dag=a^\dag b^\dag$.
\end{lemma}
\begin{proof} It follows from Lemma 2.1 that $a^\dag b=ba^\dag$ and $b^\dag a=ab^\dag$. As $b^*a=ab^*$ and $b^*a^*=a^*b^*$, then $b^*a^\dag=a^\dag b^*$, which together with $b a^\dag =a^\dag b$ imply $a^\dag b^\dag=b^\dag a^\dag$. Note that $aa^\dag$ commutes with $b$ and $b^\dag$. Also, $bb^\dag$ commutes with $a$ and $a^\dag$. Hence, $b^\dag a^\dag$ satisfies four equations of Penrose. Indeed, we have

(i) $(ab b^\dag a^\dag)^*=(aba^\dag b^\dag)^*=(aa^\dag bb^\dag)^*=bb^\dag aa^\dag= aa^\dag bb^\dag=a ba^\dag b^\dag=a bb^\dag a^\dag$.

(ii) $(b^\dag a^\dag ab)^*=(b^\dag ba^\dag a)^*=a^\dag a b^\dag b=b^\dag a^\dag a b$.

(iii) $abb^\dag a^\dag ab=aa^\dag bb^\dag ab=aa^\dag bb^\dag ba=aa^\dag ba=aa^\dag ab=ab$.

(iv) $b^\dag a^\dag ab b^\dag a^\dag=b^\dag ba^\dag a b^\dag a^\dag=b^\dag ba^\dag a a^\dag b^\dag= b^\dag ba^\dag b^\dag=b^\dag a^\dag$.

Therefore, $ab\in R^\dag$ and $(ab)^\dag= b^\dag a^\dag=a^\dag b^\dag$.
\end{proof}

Penrose \cite[p. 408]{P} presented the MP-inverse of $A+B$, where $A$ and $B$ are complex matrices such that $A^*B=0$ and $AB^*=0$. His formula indeed holds in a ring with involution.

\begin{lemma}\label{a+b} Let $a,b\in R^\dag$ such that $a^*b=ab^*=0$. Then $(a+b)^\dag=a^\dag+b^\dag$.
\end{lemma}

\section{Main results}

We say that an element $p$ is a projector if $p^2=p=p^*$. Throughout this paper, the elements $p$, $q$ are projectors from the ring $R$.

\begin{theorem}\label{ap+b(1-p)} Let $a,b\in R^\dag$ with $a^*p=pa^*$ and $b^*p=pb^*$. Then $ap+b(1-p)\in R^\dag$ and $(ap+b(1-p))^\dag=a^\dag p+ b^\dag(1-p)$.
\end{theorem}
\begin{proof} As $a^*p=pa^*$, then $ap=pa$ since $p$ is a projector. Similarly, $bp=pb$. We have $(ap)^*b(1-p)=0$. Indeed, $(ap)^*b(1-p)=pa^*(1-p)b=a^*p(1-p)b=0$. Also, $ap(b(1-p))^*=0$. By Lemma \ref{ab,a'b}, it follows that $(ap)^\dag=a^\dag p$ and $(b(1-p))^\dag=b^\dag (1-p)$. In view of Lemma \ref{a+b}, we obtain $ap+b(1-p)\in R^\dag$ and $(ap+b(1-p))^\dag=a^\dag p+ b^\dag(1-p)$.
\end{proof}

Recall from \cite{KP} that an element $a\in R$ is $*$-cancellable if $a^*ax=0$ implies $ax=0$ and $xaa^*=0$ implies $xa=0$. In a $C^*$-algebra, every element is $*$-cancellable. A ring $R$ is called $*$-reducing ring if all elements in $R$ are $*$-cancellable. 

We get the following result, under the condition of $*$-cancellabilities of some elements, rather than $*$-reducing rings in \cite{ZZCW}.

\begin{proposition} \label{equ} Let $p(1-q)$ and $q(1-p)$ be $*$-cancellable. Then the following conditions are equivalent{\rm :}\\
 $(1)~ 1-pq\in R^\dag$, $(2)~ 1-pqp\in R^\dag$, $(3)~ p-pqp\in R^\dag$, $(4)~ p-pq\in R^\dag$, $(5)~ p-qp\in R^\dag$,\\
 $(6)~ 1-qp\in R^\dag$, $(7)~ 1-qpq\in R^\dag$, $(8)~ q-qpq\in R^\dag$, $(9)~ q-qp\in R^\dag$, $(10)~ q-pq\in R^\dag$.
\end{proposition}

\begin{proof} As $a\in R^\dag \Leftrightarrow a^*\in R^\dag$, then $(1)\Leftrightarrow(6)$ and $(4)\Leftrightarrow(5)$. Also, as $p$ and $q$ play symmetric roles and  $(1)\Leftrightarrow(2)$ by \cite[Theorem 4]{ZZCW}, it is then sufficient to prove that $(2)\Leftrightarrow(3)\Leftrightarrow(4)$.

$(2)\Rightarrow(3)$ Noting $p-pqp=p(1-pqp)=(1-pqp)p$, it is an immediate result of Lemma \ref{ab,a'b}.

$(3)\Rightarrow(2)$ Since $1-pqp=p(p-pqp)+1-p$ and $(p-pqp)^*=p-pqp$, it follows from Theorem \ref{ap+b(1-p)} that $1-pqp\in R^\dag$.

$(3)\Leftrightarrow(4)$  Note that $a\in R^\dag\Leftrightarrow aa^*\in R^\dag$ and $a$ is $*$-cancellable by \cite[Theorem 5.4]{KP}. As $p(1-q)(p(1-q))^*=p-pqp\in R^\dag$ and $p-pq$ is $*$-cancellable, the result follows.
\end{proof}

Recall that an element $a\in R$ is normal if $aa^*=a^*a$. Further, if a normal element $a$ is MP-invertible, then $aa^\dag=a^\dag a$ by Lemma \ref{ab,a'b}.

In 2004, Koliha, Rako\v{c}evi\'{c} and  Stra\v{s}kraba \cite{KRS} showed that $p-q$ is nonsingular if and only if $1-pq$ and $p+q-pq$ are both nonsingular, for projectors $p$, $q$ in complex matrices, which is a $*$-cancellable ring. It is natural to consider whether the same property can be inherited to the MP-inverse in a ring with involution. The following result illustrates its possibility.

\begin{theorem} \label{equivalence} Let $p-q$, $p(1-q)$ and $q(1-p)$ be $*$-cancellable. Then the following conditions are equivalent{\rm :}

$(1)$ $p-q\in R^\dag$,

$(2)$ $1-pq\in R^\dag$,

$(3)$ $p+q-pq\in R^\dag$.
\end{theorem}

\begin{proof}  $(1)\Rightarrow(2)$ Note that $p-q$ is normal. It follows from Lemma \ref{ab,a'b} that $((p-q)^2)^\dag=((p-q)^\dag)^2$. As $p(p-q)^2=(p-q)^2p=p-pqp$, then $1-pqp=(p-q)^2 p+1-p$ and hence $1-pqp\in R^\dag$ according to Theorem \ref{ap+b(1-p)}. So, $1-pq\in R^\dag$ by \cite[Theorem 4]{ZZCW}.

$(2)\Rightarrow(1)$ By \cite[Theorem 4]{ZZCW}, we know that $1-pq\in R^\dag$ implies $1-pqp\in R^\dag$. Let $\overline{p}=1-p$ and $\overline{q}=1-q$. Note that $p(1-q)$ is $*$-cancellable. We have $1-pq\in R^\dag\Rightarrow p-pq=\overline{q}-\overline{p}~\overline{q}\in R^\dag$ by $(1) \Rightarrow (4)$ in Proposition \ref{equ}. Also, as $\overline{q}(1-\overline{p})=p(1-q)$ is $*$-cancellable, then $\overline{q}-\overline{p}~\overline{q}\in R^\dag$ implies $1-\overline{q}~\overline{p}\in R^\dag$ by $(10)\Rightarrow(6)$ in Proposition \ref{equ}, which means $1-\overline{p}~\overline{q}\in R^\dag$ since $a\in R^\dag\Leftrightarrow a^*\in R^\dag$. Again, applying \cite[Theorem 4]{ZZCW}, it follows that $1-\overline{p}~\overline{q}~\overline{p}\in R^\dag$.

Setting $a=1-pqp$ and $b=1-\overline{p}~\overline{q}~\overline{p}$, then $a^*p=pa^*$ and $b^*p=pb^*$. Since $(p-q)^2=ap+b(1-p)$, we obtain $(p-q)^2=(p-q)(p-q)^*\in R^\dag$ by Theorem \ref{ap+b(1-p)} and hence $p-q\in R^\dag$ from \cite[Theorem 5.4]{KP}.

$(1)\Leftrightarrow(3)$ In $(1)\Leftrightarrow(2)$, replacing $p$, $q$ by $1-p$, $1-q$, respectively.
\end{proof}

Next, we mainly consider the representations of the MP-inverse by aforementioned results.

\begin{theorem} \label{FGH} Let $p-q\in R^\dag$. Define $F$, $G$ and $H$ as
\begin{center}
$F=p(p-q)^\dag$, $G=(p-q)^\dag p$, $H=(p-q)(p-q)^\dag$.
\end{center}
Then, we have

$(1)$ $F^2=F=(p-q)^\dag (1-q)$,

$(2)$ $G^2=G=(1-q)(p-q)^\dag$,

$(3)$ $H^2=H=H^*$.
\end{theorem}

\begin{proof} (1) We first prove $F=(p-q)^\dag (1-q)$.

As $(p-q)^*=p-q$ and $p-q\in R^\dag$, then $(p-q)^2\in R^\dag$ by Lemma \ref{ab,a'b}. Moreover, $((p-q)^2)^\dag=((p-q)^\dag)^2$. Also, $(p-q)(p-q)^\dag=(p-q)^\dag(p-q)$. From $p(p-q)^2=(p-q)^2p$ and $p((p-q)^2)^*=((p-q)^2)^*p$, we have $p((p-q)^\dag)^2=((p-q)^\dag)^2p$ using Lemma 2.1.

Hence,
\begin{eqnarray*} (p-q)^\dag(1-q)&=& ((p-q)^\dag)^2(p-q)(1-q)=((p-q)^\dag)^2p(1-q)\\
&=&((p-q)^\dag)^2p(p-q)=p((p-q)^\dag)^2(p-q)\\
&=& p(p-q)^\dag\\
&=&F.
\end{eqnarray*}

We now show $F^2=F$. Since $p(p-q)^\dag=(p-q)^\dag(1-q)$, one can get
\begin{eqnarray*}
F^2&=&(p-q)^\dag(1-q)p(p-q)^\dag\\
&=&(p-q)^\dag(1-q)(p-q)(p-q)^\dag\\
&=&p(p-q)^\dag(p-q)(p-q)^\dag\\
&=&p(p-q)^\dag\\
&=&F.
\end{eqnarray*}

(2) By $F^*=G$.

(3) It is trivial.
\end{proof}

Under the same symbol in Theorem \ref{FGH}, more relations among $F$, $G$ and $H$ are given in the following result.

\begin{corollary} \label{FGH1} Let $p-q\in R^\dag$. Then

$(1)$ $q(p-q)^\dag=(p-q)^\dag(1-p),$

$(2)$ $qH=Hq,$

$(3)$ $G(1-q)=(1-q)F.$
\end{corollary}

\begin{proof} (1) can be obtained by a similar proof of Theorem \ref{FGH}(1).

(2) Taking involution on (1), it follows that $(1-p)(p-q)^\dag=(p-q)^\dag q$ and hence
\begin{eqnarray*}
qH&=&q(p-q)(p-q)^\dag=q(p-1)(p-q)^\dag\\
&=&-q(p-q)^\dag q=-(p-q)^\dag(1-p)q\\
&=&-(p-q)^\dag(q-p)q\\
&=&Hq.
\end{eqnarray*}

(3) We have
\begin{eqnarray*}
G(1-q)&=&(p-q)^\dag (p-q)(1-q)=(p-q)^\dag p(p-q)\\
&=&(1-q)(p-q)^\dag (p-q)\\
&=&(1-q)F.
\end{eqnarray*}
\end{proof}

Keeping in mind the relations in Theorem \ref{FGH} and Corollary \ref{FGH1}, we give the following equalities, where $\overline{a}$ denotes $1-a$.

\begin{corollary} \label{FGH2} Let $p-q\in R^\dag$. Then

$(1)$ $Fp=pG=pH=Hp,$

$(2)$ $qHq=qH=Hq=HqH$,

$(3)$ $\overline{q}\overline{F}=\overline{G}\overline{q}=\overline{q}\overline{F}\overline{q}$,

$(4)$ $(p-q)^\dag=F+G-H$.
\end{corollary}

In general, $p-q\in R^\dag$ can not imply $p + q\in R^\dag$. Such as, take $R=\mathbb{Z}$ and $1=p=q\in R$, then $p-q=0\in R^\dag$, but $p+q=2\notin R^\dag$ since 2 is not invertible.

\begin{theorem} \label{2inv} Let $2$ be invertible in $R$. Then the following conditions are equivalent{\rm :}

\rm {(1)} $pH=p$,

\rm {(2)} $(p+q)H=(p+q)$,

\rm {(3)} $p+q\in R^\dag$ and $(p+q)^\dag=(p-q)^\dag(p+q)(p-q)^\dag$.
\end{theorem}

\begin{proof}

$(1)\Rightarrow(2)$ If $pH=p$, then $qH=q$ by the symmetry of $p$ and $q$. Hence $(p+q)H=(p+q)$.

$(2)\Rightarrow(1)$ Note that $H=(p-q)(p-q)^\dag$ and $p-q$ is normal. We have $(p-q)H=p-q$ and $p+q=(p+q)H=(q-p)H+2pH=-(p-q)+2pH$, which implies $2pH=2p$. Hence, $pH=p$ since $2$ is invertible.

$(2)\Rightarrow(3)$ Let $x=(p-q)^\dag(p+q)(p-q)^\dag$. We prove that $x$ is the MP-inverse of $p+q$ by checking four equations of Penrose.

(i) $((p+q)x)^*=(p+q)x$. Indeed,
\begin{eqnarray*}
(p+q)x&=&(p+q)(p-q)^\dag(p+q)(p-q)^\dag\\
&=& (p-q)^\dag(1-q+1-p)(p+q)(p-q)^\dag\\
&=& (p-q)^\dag(p-q)^2(p-q)^\dag\\
&=& (p-q)(p-q)^\dag.
\end{eqnarray*}

(ii) $(x(p+q))^*=x(p+q)$. By similar proof of (i), we have $x(p+q)=(p-q)^\dag(p-q)$.

(iii) Note that the relations $pH=Hp$ and $qH=Hq$ in Corollary \ref{FGH2}. Then
\begin{eqnarray*}
(p+q)x(p+q)&=&(p-q)(p-q)^\dag(p+q)\\
&=& H(p+q)=(p+q)H\\
&=&p+q.
\end{eqnarray*}

(iv) It follows that $x(p+q)x=(p-q)^\dag(p+q)(p-q)^\dag (p-q)(p-q)^\dag=x$.

$(3)\Rightarrow(2)$ As $p+q\in R^\dag$ with $(p+q)^\dag=(p-q)^\dag(p+q)(p-q)^\dag$, then
\begin{eqnarray*}
p+q&=&(p+q)(p+q)^\dag (p+q)=(p+q)(p-q)^\dag(p+q)(p-q)^\dag (p+q) \\
&=& (p+q)(p-q)^\dag(p-q)^\dag (1-q+1-p)(p+q)\\
&=&(p+q)(p-q)^\dag(p-q)^\dag [(1-q)p+(1-p)q]\\
&=&(p+q)(p-q)^\dag(p-q)^\dag [(p-q)p+(q-p)q]\\
&=&(p+q)(p-q)^\dag(p-q)^\dag (p-q)p- (p+q)(p-q)^\dag(p-q)^\dag(p-q)q\\
&=&(p+q)(p-q)^\dag(p-q)(p-q)^\dag p- (p+q)(p-q)^\dag(p-q)(p-q)^\dag q\\
&=&(p+q)(p-q)^\dag p- (p+q)(p-q)^\dag q\\
&=&(p+q)(p-q)^\dag (p-q)\\
&=&(p+q)H.
\end{eqnarray*}
\end{proof}

Next, we give a new necessary and sufficient condition of the existence of $(p+q)^\dag$, where $p$ and $q$ commute.

\begin{theorem} Let $p,q\in R$ with $pq=qp$. Then $p+q\in R^\dag$ if and only if $1+pq\in R^\dag$.

In this case, $(p+q)^\dag=(1+pq)^\dag p+q(1-p)$ and $(1+pq)^\dag=(p+q)^\dag p+1-p$.
\end{theorem}

\begin{proof} Suppose $p+q\in R^\dag$. As $1+pq=p(p+q)+1-p$, then $(1+pq)^\dag=(p+q)^\dag p+1-p$ by Theorem \ref{ap+b(1-p)}.

Conversely, let $x=(1+pq)^\dag p+q(1-p)$. We next show that $x$ is the MP-inverse of $p+q$.

(i) $[(p+q)x]^*=(p+q)x$. We have
\begin{eqnarray*}
(p+q)x&=&(p+q)[(1+pq)^\dag p+q(1-p)]\\
&=&(1+pq)^\dag p+(1+pq)^\dag pq+q(1-p)\\
&=& (1+pq)^\dag (1+pq) p+q(1-p).
\end{eqnarray*}

Hence, $[(p+q)x]^*=(p+q)x$.

(ii) It follows that $[x(p+q)]^*=x(p+q)$ since $p$ and $q$ commute.

(iii) $(p+q)x(p+q)=p+q$. Indeed,
\begin{eqnarray*}
(p+q)x(p+q)&=& (p+q)[(1+pq)^\dag (1+pq) p+q(1-p)]\\
&=&(1+pq)^\dag (1+pq) p+(1+pq)^\dag (1+pq) pq+q(1-p)\\
&=&(1+pq)^\dag (1+pq)p(1+pq)+q(1-pq)\\
&=&p(1+pq)+q(1-pq)\\
&=& p+q.
\end{eqnarray*}

(iv) By a similar way of (3), we get $x(p+q)x=x$.

Thus, $(p+q)^\dag=(1+pq)^\dag p+q(1-p)$.
\end{proof}

The next theorem, a main result of this paper, admits proficient skills on $F$, $G$ and $H$, expressing the formulae of the MP-inverse of difference of projectors.

\begin{theorem} \label{some equ} Let $p-q\in R^\dag$. Then

$(1)$ $(1-pqp)^\dag=p((p-q)^\dag)^2+(1-p)$,

$(2)$ $(1-pq)^\dag=p((p-q)^\dag)^2-pq(p-q)^\dag+1-p$,

$(3)$ $(p-pqp)^\dag=p((p-q)^\dag)^2$,

$(4)$ If $p-pq$ is $*$-cancellable, then $(p-pq)^\dag=(p-q)^\dag p$,

$(5)$ If $p-pq$ is $*$-cancellable, then $(p-qp)^\dag=p(p-q)^\dag$.
\end{theorem}

\begin{proof} (1) As $1-pqp=p(p-q)^2+1-p$, then $(1-pqp)^\dag=p((p-q)^\dag)^2+1-p$ according to Theorem \ref{ap+b(1-p)}.

(2) It follows from Theorem \ref{equivalence} that $p-q\in R^\dag$ implies $1-pq\in R^\dag$. Let $x=p((p-q)^\dag)^2-pq(p-q)^\dag+1-p$. We next show that $x$ is the MP-inverse of $1-pq$.

(i) We have
\begin{eqnarray*}
(1-pq)x&=&(1-pq)[p((p-q)^\dag)^2-pq(p-q)^\dag+1-p]\\
&=&(p-pqp)((p-q)^\dag)^2-(1-pq)pq(p-q)^\dag+(1-pq)(1-p)\\
&=& p(p-q)^2((p-q)^\dag)^2-(p-pqp)(p-q)^\dag (1-p)+(1-pq)(1-p)\\
&=&p(p-q)(p-q)^\dag-p(p-q)^2(p-q)^\dag(1-p)+(1-pq)(1-p)\\
&=&p(p-q)(p-q)^\dag-p(p-q)(1-p)+(1-pq)(1-p)\\
&=&p(p-q)(p-q)^\dag+1-p\\
&=&pH+1-p.
\end{eqnarray*}

Hence, $((1-pq)x)^*=(1-pq)x$ since $pH=Hp$ and $H^*=H$.

(ii) We get $x(1-pq)=p(p-q)^\dag p+1-p$. Hence, $(x(1-pq))^*=x(1-pq)$.

(iii) $(1-pq)x(1-pq)=1-pq$. Indeed,
\begin{eqnarray*}
(1-pq)x(1-pq)&=& (pH+1-p)(1-pq)= Hp(1-pq)+(1-p)(1-pq)\\
&=&Hp(p-pq)+1-p=pH(p-pq)+1-p\\
&=&pHp(p-q)+1-p=pH(p-q)+1-p\\
&=&p(p-q)+1-p\\
&=&1-pq.
\end{eqnarray*}

(iv) $x(1-pq)x=1-pq$. Actually, we can obtain this result by a similar proof of (iii).

(3) Since $p-pqp=p(p-q)^2=(p-q)^2p$, we get $(p-pqp)^\dag=p((p-q)^\dag)^2$ by Lemma \ref{ab,a'b}.

(4) Keeping in mind that $a^\dag=a^*(aa^*)^\dag=(a^*a)^\dag a^*$, we have $(p-pq)^\dag=(p-qp)p((p-q)^\dag)^2=(p-q)((p-q)^\dag)^2p=(p-q)^\dag p$.

(5) Note that $a$ is $*$-cancellable if and only if $a^*$ is $*$-cancellable. It follows from $(a^*)^\dag=(a^\dag)^*$ that $(p-qp)^\dag=p(p-q)^\dag$.
\end{proof}

\begin{corollary} Let $p-pq$ be $*$-cancellable and let $1-pq\in R^\dag$. Then $p-q\in R^\dag$ and
\begin{center}
$(p-q)^\dag=(1-pq)^\dag (p-pq)+(p+q-pq)^\dag(pq-q)$.
\end{center}
\end{corollary}

\begin{proof} From Theorem \ref{equivalence}, we have $p-q\in R^\dag\Leftrightarrow 1-pq\in R^\dag$.

By Theorem \ref{some equ} (2), we have $(p+q-pq)^\dag=(1-p)((p-q)^\dag)^2+(1-p)(1-q)(p-q)^\dag+p$.
It is straight to check that $(1-pq)^\dag (p-pq)+(p+q-pq)^\dag(pq-q)$ satisfies four equations of Penrose.
\end{proof}

The following result is motivated by \cite{D0}, therein, Deng considered the Drazin inverses of difference of idempotent bounded operators on Hilbert spaces.

\begin{theorem} Let $pq-qp$ be $*$-cancellable. Then

$(1)$ $(p-q)^\dag=p-q$ if and only if $pq=qp$,

$(2)$ If $6$ is invertible in $R$, then $(p+q)^\dag=p+q$ if and only if $pq=0$.
\end{theorem}

\begin{proof} (1) If $pq=qp$, it is straightforward to check $(p-q)^\dag=p-q$.

Conversely, $(p-q)^\dag=p-q$ implies $(p-q)^3=p-q$, we get $pqp=qpq$ and hence $(pq-qp)^*(pq-qp)=0$. It follows that $pq=qp$ since $pq-qp$ is $*$-cancellable.

(2) Suppose $pq=0$. Then $p^*q=pq^*=0$ since $p$, $q$ are projectors. Then $(p+q)^\dag=p+q$ by Lemma 2.3.

Conversely, $(p+q)^\dag=p+q$ concludes $(p+q)^3=p+q$. By direct calculations, it follows that $2pq+2qp+pqp+qpq=0$. \hfill(3.1)

Multiplying the equality (3.1) by $p$ on the left yields
$2pq+3pqp+pqpq=0$. \hfill(3.2)

Multiplying the equality (3.1) by $q$ on the right gives
$2pq+3qpq+pqpq=0$. \hfill(3.3)

Combining the equalities $(3.2)$ and $(3.3)$, it follows that $pqp=qpq$ since 3 is invertible. As $pq-qp$ is $*$-cancellable, then $pqp=qpq$ implies $pq=qp$.
Hence, equality $(3.1)$ can be reduced to $6pq=0$.

Thus, $pq=0$.
\end{proof}

\begin{theorem} Let $1-p-q\in R^\dag$. Then

$(1)$ $pqp\in R^\dag$ and $(pqp)^\dag=p((1-p-q)^\dag)^2=((1-p-q)^\dag)^2p$,

$(2)$ If $pq$ is $*$-cancellable, then $pq\in R^\dag$ and $(pq)^\dag=qp((1-p-q)^\dag)^2$.
\end{theorem}

\begin{proof} (1) Since $(1-p-q)^*=1-p-q$, we have $((1-p-q)^2)^\dag=((1-p-q)^\dag)^2$ by Lemma \ref{ab,a'b}. As $pqp=p(1-p-q)^2=(1-p-q)^2p$, then $pqp\in R^\dag$ from Lemma \ref{ab,a'b} and hence $(pqp)^\dag=p((1-p-q)^\dag)^2=((1-p-q)^\dag)^2 p$.

(2) Note that $1-p-q\in R^\dag$ implies $pqp\in R^\dag$. As $pqp=pq(pq)^*$ and $pq$ is $*$-cancellable, then $pq\in R^\dag$ by \cite[Theorem 5.4]{KP}. The formula $a^\dag=a^*(aa^*)^\dag$ guarantees that $(pq)^\dag=qp((1-p-q)^\dag)^2$.
\end{proof}

\centerline {\bf ACKNOWLEDGMENTS} The authors are highly grateful to the anonymous referee for his/her valuable comments
which led to improvements of the paper. The first author is grateful to China Scholarship Council for supporting him to purse his further study in University of Minho, Portugal. This research is supported by the National Natural Science Foundation of China (No. 11371089), the Specialized Research Fund for the Doctoral Program of Higher Education (No. 20120092110020), the Natural Science Foundation of Jiangsu Province (No. BK20141327), the Foundation of Graduate Innovation Program of Jiangsu Province(No. CXLX13-072), the Scientific Research Foundation of Graduate School of Southeast University, the FEDER Funds through ¡®Programa Operacional Factores de Competitividade-COMPETE' and the Portuguese Funds through FCT- `Funda\c{c}\~{a}o para a Ci\^{e}ncia e Tecnologia', within the project PEst-OE/MAT/UI0013/2014.

\end{document}